\documentclass[a4paper,10pt]{amsart}
\usepackage{amsmath,amsthm,amssymb,tikz,calc,graphicx}
\usepackage[latin1]{inputenc}
\usepackage[T1]{fontenc}
\usepackage[all]{xy}
\usepackage{mathrsfs,pifont}
\usepackage{booktabs}
\usepackage{array}
\usepackage{longtable}
\usepackage{placeins}

\usepackage[hidelinks]{hyperref}

\newtheorem*{thm*}{Main Theorem}

\newtheorem{cor}[equation]{Corollary}
\newtheorem{lem}[equation]{Lemma}
\newtheorem{prop}[equation]{Proposition}

\theoremstyle{definition}

\newtheorem{defn-lem}[equation]{Definition-Lemma}

\newtheorem*{rem*}{Remark}

\numberwithin{equation}{section}

\usepackage{bm}
\usepackage{upgreek}
\makeatletter
\newcommand{\bfgreek}[1]{\bm{\@nameuse{up#1}}}

\newcommand{\triv}{{\textbf{1}}}
\def\R{\mathbb R}

\def\Z{\mathbb Z}
\def\A{\mathbb A}
\def\Q{\mathbb Q}
\def\C{\mathbb C}

\def\F{\mathbb F}

\def\hra{\hookrightarrow}
\def\ra{\rightarrow}

\def\q{\mathfrak q}

\def\O{\mathcal O}

\def\p{\mathfrak p}

\def\<{\langle}
\def\>{\rangle}

\def\Gal{{\rm Gal}}

\makeatletter
\DeclareRobustCommand\smallop[2][1]{%
	\mathop{\vphantom{\oplus}\mathpalette\smallop@{{#1}{#2}}}\slimits@
}
\newcommand{\smallop@}[2]{\smallop@@#1#2}
\newcommand{\smallop@@}[3]{%
	\vcenter{%
		\sbox\z@{$#1\oplus$}%
		\hbox{\resizebox{\ifx#1\displaystyle#2\fi\dimexpr\ht\z@+\dp\z@}{!}{$\m@th#3$}}%
	}%
}
\makeatother

\makeatletter
\DeclareRobustCommand\bigop[2][1]{%
	\mathop{\vphantom{\bigoplus}\mathpalette\bigop@{{#1}{#2}}}\slimits@
}
\newcommand{\bigop@}[2]{\bigop@@#1#2}
\newcommand{\bigop@@}[3]{%
	\vcenter{%
		\sbox\z@{$#1\bigoplus$}%
		\hbox{\resizebox{\ifx#1\displaystyle#2\fi\dimexpr\ht\z@+\dp\z@}{!}{$\m@th#3$}}%
	}%
}
\makeatother

\begin{document}

\title{All totally real abelian number fields of class number one are Euclidean}
\author{Harald Grobner}
\thanks{The author has been supported by the research projects PAT-4628923 and PAT-2584625 of the Austrian Science Fund (FWF), the EU NextGeneration project (IIP\_UNIPU\_010159) and the Croatian Science Foundation project (HRZZ-IP-2022-10-4615).}
\address{Harald Grobner, Faculty of Mathematics, University of Vienna, Oskar-Morgenstern-Platz 1, A-1090 Vienna, Austria}
\email{harald.grobner@univie.ac.at}

\date{\today}

\maketitle

%\today
 
\begin{abstract}
It is a famous result of Narkiewicz that there are at most two real quadratic and at most one totally real cubic Galois extensions $K/\Q$, whose ring of integers $\O_K$ is a unique factorization domain (UFD), but not Euclidean. As predicted by another famous result due to Weinberger, these exceptions should in fact not exist, as their existence would contradict a version of the Extended Riemann hypothesis. Building on recent spectacular progress on the (Quasi-)Riemann hypothesis for Hecke characters, in this short note we close the above final gap provided by Narkiewicz's result, and show, more generally, that the ring of integers $\O_K$ of any totally real abelian number field $K$ is a UFD, if and only if it is Euclidean. By referring to an additional announced uniform zero-free half-plane for general Dedekind zeta functions, we also establish the corresponding result for certain explicit new families of non-Galois cubic, quartic and quintic fields of unit rank one or two.
\end{abstract}

\vskip 10pt

\setcounter{tocdepth}{1}
%\tableofcontents

\section*{Introduction}

\noindent About 20 years ago, Narkiewicz proved in a short paper, cf.\ \cite{nark}, that there are at most two quadratic and at most one cubic totally real Galois extensions $K/\Q$, whose ring of integers $\O_K$ is a unique factorization domain (UFD), but not Euclidean. As $\O_K$ is Dedekind, this amounts to the dictum that there are at most two real quadratic and at most one totally real cubic Galois extensions $K/\Q$, such that $\mathcal O_K$ is a principal ideal domain (PID), i.e., $K$ has class number $h_K=1$, but there is no function $\omega:\O_K\setminus\{0\}\ra\Z_{\geq 0}$ such that for every $a,b\in\O_K$, $b\neq 0$, there exist $q,r\in\O_K$ such that $a=bq+r$ and either $r=0$ or $\omega(r)<\omega(b)$.\\\\
When Narkiewicz wrote his famous note, it had been widely known, that such exceptional extensions $K/\Q$ should not exist, as their existence would contradict a version of the Extended Riemann hypothesis (ERH), which is an important result of Weinberger, \cite{wein}. \\\\
In this short note we verify this prediction of Weinberger's result and prove the following unconditional result for all totally real abelian number fields of degree at least two:

\begin{thm*}
Let $K/\Q$ be a totally real abelian number field with $n=[K:\Q]\geq2$, ring of integers $\O_K$ and class number $h_K$. For $x>1$, let $\Pi_K(x)$ be the number of non-zero prime ideals $\p$ of $K$, whose norm $N(\p)$ is a rational prime strictly between $x$ and $2x$, i.e., $x<N(\p)<2x$, and for which the natural map $\O_K^*\ra(\O_K/\p)^*$ is surjective.
Then there are constants $c_K,x_K>0$ such that
$$\Pi_K(x) \geq c_K \cdot \frac{x}{\log(x)^2}$$
for all $x>x_K$. In particular, for all totally real abelian number fields $K$ 
$$\O_K\text{ is Euclidean}
 \quad\Longleftrightarrow\quad h_K=1.$$
\end{thm*}

\noindent We would like to point out that (as of the 4th of December 2026), this result of Weinberger's is still conditional, even if one takes into account the most recent available claimed progress on the ERH, cf. \cite{Liu26-ERH,Axiom26}. Indeed, the main result of \cite{wein} uses a version of ERH stated for infinite families of Dirichlet $L$-functions attached to the splitting fields $L_\ell=K(\zeta_\ell,\sqrt[\ell]{\varepsilon})$ of $X^\ell-\varepsilon$ over $K$, as $\ell$ ranges over the primes and $\varepsilon\in \O_K^*$ is a fixed unit: Even when $K$ is totally real, $L_\ell$ in general does not have to be a CM field (take $K=\Q(\sqrt 5)$, $\varepsilon=(1+\sqrt5)/2$ the {\it Golden ratio} and $\ell$ any odd prime), whence does not come under the purview of the above mentioned claimed progress on ERH.\\\\The reader will immediately observe that the conclusion in our Main Theorem about $\O_K$ being Euclidean if and only if $K$ has class number $h_K=1$ indeed follows from (the fact that $\O_\Q=\Z$ is Euclidean and) the first part of our theorem by a well-known and fundamental result of Harper-Murty, cf.\ \cite{HM}, Prop.\ 2 -- a way of reasoning, which goes exactly along the same lines as the arguments in \cite{nark} as well as along the ones in the just quoted \cite{HM}, but also in the more recent \cite{MP}. In fact, our theorem overlaps and partly reproves some important results in all of the aforementioned sources in the case of totally real abelian extensions $K/\Q$ -- we particularly mention \cite{HM}, Thm.\ 2 ($K/\Q$ abelian with $rk_\Z(\O_K)\geq 3$), and \cite{MP}, Cor.\ 1.5, ($K/\Q$ totally real, Galois over some proper subfield and $[K:\Q]\geq5$) -- whereas it establishes the non-existence of the hypothetical exceptional quadratic, resp.\ cubic extensions of the aforementioned main result of \cite{nark}, as every cubic Galois extension of $\Q$ is clearly necessarily totally real and abelian.\\\\
Having said this, the novelty underlying our Main Theorem is hence contained in showing the lower bound $c_K \frac{x}{\log(x)^2}$ for $\Pi_K(x)$, $x\gg 0$, when $K/\Q$ is totally real abelian of degree $[K:\Q]=2,3$.\\\\ To this end, our main ingredients are two results from the recent source \cite{OAI}, Thm.\ 1.2 and Prop.\ 2.1, {\it ibidem}. Once verified and part of the {\it a priori} knowledge of mankind, they establish the following two important assertions, which shall be treated as a blackbox here:\\\\
\noindent
\textbf{Scholium (I)}: {\it There is a $\delta_0\in(0,1/2)$ such that, for every cyclotomic field $F$ containing the group of 12th roots of unity $\mu_{12}$ and every Hecke character $\chi:\A^*_F/F^*\ra\C^*$ of $F$ of finite order, the associated Hecke $L$-function $L(s,\chi)$ has no zeros in the open region
$$\Re e(s)>1-\delta_0.$$}
\noindent Here, the width $\delta_0$ is in fact independent of $F$ and $\chi$. Thm.\ 1.2 of \cite{OAI} states that Scholium (I) is true with uniform width $\delta_0=10^{-6}$, hence establishing a uniform version of a Quasi-Riemann hypothesis for all finite-order Hecke characters of all cyclotomic fields containing $\mu_{12}$. The next assertion is far more technical in nature and borrowed from the same source:\\\\
\noindent
\textbf{Scholium (II)}: {\it For every integer $k\geq 1$, there are $a_k, x_k>0$ such that for every $x\geq x_k$, at least $a_k x/\log(x)^2$ distinct rational primes $p$ with $x<p<2x$ satisfy
\begin{equation}\label{eq:CP}
 p\equiv-1\pmod{8k},\quad {\rm and} \quad p-1=2rq,
\end{equation}
where $q>x^{9/10}$ is prime and $r$ is a positive integer, whose prime divisors $\ell$ are all in between $\exp(\log(x)^{1/10})$ and $\exp(\log(x)^{3/10})$.}\\\\
\noindent This statement is the case $M=8k$, $c=2$ and $u=-1$ in Prop.\ 2.1 of \cite{OAI}.\\\\ 
In Prop.\ \ref{lem:smooth} below, we establish a refined, as well as generalized version of another result in \cite{OAI} (its Lem.\ 9.2). Our more general approach here is in fact guided by our recent work on the prime-counting function $\pi(x)$, cf.\ \cite{grobner}. To state Prop.\ \ref{lem:smooth}, let $x>1$ be any real number, let $F$ be any number field whose absolute discriminant is denoted $D_F$ and let $\varphi\in C_c^\infty(\R_{>0}, \R_{\geq 0})$ be any smooth function of compact support. Then,
$$\Theta_{F,\varphi}(x):=\sum_{\mathfrak q}\sum_{j\geq1}\log(N(\mathfrak q))\,\varphi\!\left(\frac{N(\mathfrak q)^j}{x}\right) \ll_\varphi x\widehat\varphi(1)+x^{\Theta_F}(\log D_F+[F:\Q]),$$
where
$$\Theta_F:=\sup\{\Re e(\rho) \ | \ \zeta_F(\rho)=0, \ 0<\Re e(\rho)<1\}$$
is Riemann's constant for the Dedekind zeta-function $\zeta_F(s)$ of $F$ and $\widehat\varphi$ denotes the Mellin transform of $\varphi$. Here, in the above series defining $\Theta_{F,\varphi}(x)$, the first sum runs over the non-zero prime ideals $\q$ of the ring of integers $\O_F$ of $F$ and $N(\q)$ denotes their norm.\\\\ 
We remark that our Prop.\ \ref{lem:smooth} applies unconditionally to any number field (unlike \cite{OAI}, Lem.\ 9.2, which applies only to those number fields, for which a version of the Quasi-Riemann hypothesis is known) and contains \cite{OAI}, Lem.\ 9.2 as a proper corollary. As this is an important technical novelty, the reader will find a full proof below (unlike in \cite{OAI}, where a proof is only sketched, disguising the argument for why the implied bound is independent of the number fields under consideration, which is crucial in applications, however). \\\\ 
When combined with Scholia (I) and (II), it yields the desired lower bound of growth for the ``prime ideal counting function'' $\Pi_K(x)$. This final argument is carried out in two steps, to be found in our Sect.\ \ref{sec:goodprimes} and Sect.\ \ref{sect:mthmproof}, respectively, to which we refer for all details.\\\\
Finally, if we allow ourselves to refer to even another recent preprint on the Quasi-Riemann Hypothesis, Sect.\ \ref{sec:extensions} explains how to extend the above argument to certain families of number fields which need not be totally real, nor abelian. More precisely, the relevant additional inputs, necessarly for this purpose, are the announced uniform zero-free half-plane result for Dedekind zeta functions stated in \cite{Liu26}, Thm.\ 17.1, together with the full power of the just quoted \cite{OAI}, Prop.\ 2.1. If one takes these two results as given, we show how to prove the same lower bound for $\Pi_K(x)$ for any number field $K$ of positive unit rank, satisfying the following, at first maybe seemingly artificial condition:\\\\ 
{\it There are fixed, coprime integers $k\geq1$ and $u$ with $\gcd(u-1,8k)=2$, such that every sufficiently large rational prime $p\equiv u\pmod{8k}$ admits a prime ideal $\p$ of $\O_K$ with $N(\p)=p$.}\\\\
We remark that the Frobenius action on the embeddings of $K$ provides a group-theoretic way to verify this hypothesis. It suffices to find an arithmetic progression $u\bmod 8k$, satisfying the above conditions on $k,u$, which prescribes a Frobenius class in the maximal abelian quotient of the Galois group of the normal closure of $K/\Q$, such that every element of the inverse image of this class fixes at least one embedding of $K$ into its normal closure.\\\\
Most interesting are the additional new classes of number fields for which $h_K=1$ then yields that $\O_K$ is Euclidean: The resulting applications include all non-Galois complex cubic fields, all non-Galois totally real cubic fields whose quadratic resolvent differs from $\Q(\sqrt3)$, and certain quartic and quintic families. We refer the reader to our Table~\ref{tab:extensions}, where all these cases are summarized in detail. We remark that all these fields have unit rank one or two, so the applications lie below the unit-rank threshold of \cite{HM}, Prop.\ 7 (with the empty admissible set). We point out that individual non-Galois Euclidean cubic fields were already known, for instance from \cite{Clark96}. Likewise, the theorem of \cite{SSS23} for imaginary biquadratic and cyclic quartic fields concerns Galois quartic extensions, whereas both quartic rows in Table\ \ref{tab:extensions} are non-Galois. 
\small  
%\noindent{\it Acknowledgments:} %The final version of this article was compiled with the support of AI-tools: Aside from proofreading the latex-code, we used them as some sort of a ``sparring partner'' in some useful discussions -- in much the same way as we would have discussed with a human colleague, whose name would now be acknowledged. Needless to say that we assume full responsibility for the contents of the paper.
 
%The following is a quote from ChatGPT, which suggested the following acknowledgments, in order to fairly explain its own role in the development of this paper: "The author used ChatGPT (OpenAI) as a mathematical sparring partner during the development of this work, in much the same way as one might discuss developing arguments with a colleague. These discussions were used to test proof strategies, examine intermediate arguments, and improve the exposition. The author independently verified all mathematical claims and assumes full responsibility for the contents of the paper." I think this is exactly recovered by my own version, whence I find it appropriate to leave it as it is.
\normalsize

\section{Several key-observations}
\subsection{Notational preliminaries}
We use the standard Vinogradov- and Landau-notation. Thus, if $f$ is a complex-valued and $g$ is non-negative real-valued function, then $f(T)\ll g(T)$ means that there exists a constant $C>0$ such that $|f(T)|\le C\cdot g(T)$ for all sufficiently large $T$. Equivalently, $f(T)=O(g(T))$. If we want to indicate that the implied constant is allowed to depend on one or more fixed parameters, we indicate this by subscripts. For example, $f(T)\ll_{a,b,c} g(T)$ means that $|f(T)|\le C_{a,b,c}\cdot g(T)$ for all sufficiently large $T$, where $C_{a,b,c}>0$ may depend on ${a,b,c}$, but not on $T$. We write synonymously $f(T)=O_{a,b,c}(g(T))$. Analogously, if $g(T)>0$ for $T\gg 0$, then $f=o(g)$ means that $|f(T)g(T)^{-1}|\to 0$ as $T\to\infty$.\\\\
For $n\geq 2$, we write $\mu_n$ for the finite group of $n$-th roots of unity in $\C$. The norm attached to a finite extension of subfields $L/F$ of $\C$ shall be denoted by $N_{L/F}$, the norm of an ideal $\q$ of a number field $L$ simply by $N(\q)$. If $L$ denotes a number field, we shall write $\O_L$ for its ring of integers and $D_L$ for its absolute discriminant.\\\\ 
Our case of interest shall be the one of a totally real abelian extension $K/\Q$ with $n=[K:\Q]\geq2$ (where we recall that being abelian includes the assumption of being Galois). We write $h_K$ for its class number and denote its absolute discriminant simply by $D=D_K>0$. By the Kronecker--Weber theorem, cf.\ \cite{Neukirch}, V.1.10, we may fix, once and for all, an integer $m\geq1$ such that $K\subseteq\Q(\zeta_m)$. Since $K$ is totally real, we have
$$\O_K^*=\{\pm\varepsilon_1^{j_1}\cdots\varepsilon_{n-1}^{j_{n-1}} \ | \  j_1,\ldots,j_{n-1}\in\Z\},$$
where the classes of $\varepsilon_1,\ldots,\varepsilon_{n-1}$ form a basis of $\O_K^*/\{\pm1\}$. Fix such a system and put $\varepsilon=\varepsilon_1$. After replacing $\varepsilon_1$ by its inverse and/or its negative, we may assume $\varepsilon>1$ in the fixed real embedding $K\subset\R$.

\subsection{Some field extensions}
For an odd prime $\ell$, set $B_\ell=K(\zeta_\ell)$ and $H_\ell=K(\zeta_\ell,\varepsilon^{1/\ell})$, the Galois-extension $H_\ell/K$ given by the splitting field of $X^\ell-\varepsilon$ over $K$. We record

\begin{lem}\label{lem:degree} 
If $\ell\nmid D$, then $[B_\ell:K]=\ell-1$ and $[H_\ell:K]=\ell(\ell-1)$.
\end{lem}
\begin{proof} 
To see the first assertion, consider the intersection $K\cap\Q(\zeta_\ell)$. Every non-trivial subfield of $\Q(\zeta_\ell)$ is ramified at $\ell$: The prime $\ell$ is totally ramified in the cyclotomic extension, and total ramification descends to its subfields. Since $K$ is unramified at $\ell$, as $\ell\nmid D$ by assumption, the intersection is $\Q$. This proves the first degree formula $[B_\ell:K]=\ell-1$.\\\\
For the second formula, we first recall that the unit $\varepsilon$ is not an $\ell$-th power in $K$. Indeed, an element of $K$, whose $\ell$-th power is $\varepsilon$, is an algebraic integer, and its inverse is an algebraic integer as well, whence it is a unit of $\O_K$. Comparing the $\varepsilon_1$-coordinate in $\O_K^*/\{\pm1\}$ would then give the impossible equality $\ell j_1=1$ for some $j_1\in\Z$. Strengthening this assertion, we observe that it is not an $\ell$-th power in $B_\ell$ either: In fact, if $\alpha^\ell=\varepsilon$ with $\alpha\in B_\ell$, put $b=N_{B_\ell/K}(\alpha)$. Taking norms gives $b^\ell=\varepsilon^{\ell-1}$, and hence $\left(\frac{\varepsilon}{b}\right)^\ell=\varepsilon$, contradicting our preceding observation. Since $B_\ell$ contains $\mu_\ell$, the splitting field of $X^\ell-\varepsilon$ over $B_\ell$ has Galois group a subgroup of $\mu_\ell$, so its degree is either $1$ or $\ell$. The former would put an $\ell$-th root in $B_\ell$, so the degree is $\ell$ as desired. 
\end{proof}

\begin{lem}\label{lem:disc}
If $\ell\nmid D$,
$$ D_{B_\ell}\leq D^{\ell-1}\ell^{n(\ell-2)}, \quad D_{H_\ell}\leq D^{\ell(\ell-1)}\ell^{n\ell(2\ell-3)}, \quad \frac{\log D_{H_\ell}}{[H_\ell:\Q]}
 \leq\frac1n\log D+2\log\ell.$$
\end{lem}

\begin{proof}
As we have just seen above, the fields $K$ and $\Q(\zeta_\ell)$ are linearly disjoint. The order generated by their rings of integers in the compositum has discriminant $D^{\ell-1}D_{\Q(\zeta_\ell)}^n$ and the field discriminant divides this order discriminant with square quotient. Since $D_{\Q(\zeta_\ell)}=\ell^{\ell-2}$, the first inequality follows. For the second, observe that the root $\varepsilon^{1/\ell}$ is integral over $\mathcal O_{B_\ell}$ and that the polynomial discriminant of $X^\ell-\varepsilon$ is, up to sign, $\ell^\ell\varepsilon^{\ell-1}$. The second factor is a unit, so the relative discriminant ideal of $H_\ell/B_\ell$ divides $\ell^\ell\mathcal O_{B_\ell}$. The discriminant formula in a tower therefore gives
$$ D_{H_\ell}\leq D_{B_\ell}^{\ell} \ell^{\ell[B_\ell:\Q]} \leq D^{\ell(\ell-1)} \ell^{n\ell(\ell-2)+n\ell(\ell-1)},$$
which is nothing but the second inequality. Divide its logarithm by $[H_\ell:\Q]=n\ell(\ell-1)$ and use $(2\ell-3)/(\ell-1)<2$ to obtain the last claim.
\end{proof}

\begin{lem}\label{lem:splittingtest}
Let $\p$ be a prime ideal of $K$, whose norm is a prime $p=N(\p)$ with $p\nmid\ell D$. Then, $\p$ splits completely in $H_\ell/K$, if and only if
$$p\equiv1\pmod\ell,\qquad \varepsilon^{(p-1)/\ell}\equiv1\pmod\p.$$
\end{lem}

\begin{proof}
This is well-known. For the sake of completeness we sketch the argument: The prime $\p$ is unramified in $H_\ell$. Moreover, the cyclotomic extension and the polynomial discriminant above show that no prime other than one above $\ell$ can ramify relatively. The arithmetic Frobenius acts on $\zeta_\ell$ by its $p$-th power, so it fixes $B_\ell$ precisely when $p\equiv1\pmod\ell$. Under this condition, if $\alpha^\ell=\varepsilon$ and $\mathfrak P$ is a prime of $H_\ell$ above $\p$, then
$$\frac{\operatorname{Frob}_{\mathfrak P}(\alpha)}{\alpha}\in\mu_\ell,\qquad\frac{\operatorname{Frob}_{\mathfrak P}(\alpha)}{\alpha}\equiv\alpha^{p-1}=\varepsilon^{(p-1)/\ell}\pmod{\mathfrak P}.$$
where we recall that the division is legitimate because $\alpha$ is a unit. As reduction on $\mu_\ell$ is injective at $p\ne\ell$, the Frobenius fixes $\alpha$ exactly when the second congruence $\varepsilon^{(p-1)/\ell}\equiv1\pmod\p$ above holds. It then fixes the generators of $H_\ell/K$, which is the desired criterion.
\end{proof}

\subsection{Descent of the common zero-free region and a refined growth condition}\label{sec:descent}
In view of Scholium (I), the following assertion, once worth a Fields Medal, now becomes another simple
\begin{lem}\label{lem:descent}
For every odd prime $\ell$, the Dedekind zeta function $\zeta_{H_\ell}(s)$ has no zero in $\Re e(s)>1-\delta_0$ with $\delta_0=10^{-6}$.
\end{lem}

\begin{proof}
Put $F_\ell=\Q(\zeta_{12m\ell})$ and $E_\ell=F_\ell H_\ell$. Since $K\subseteq\Q(\zeta_m)$, the cyclotomic field $F_\ell$ contains $K$, $\mu_\ell$ and $\mu_{12}$. Consequently $E_\ell=F_\ell(\varepsilon^{1/\ell})$, and $E_\ell/F_\ell$ is cyclic of degree $1$ or $\ell$. The factorization of a Dedekind zeta function for an abelian extension into finite-order Hecke $L$-functions gives
$$ \zeta_{E_\ell}(s)= \prod_{\chi\in\widehat{\Gal(E_\ell/F_\ell)}} L(s,\chi).$$
Here and below group characters are identified with their class-field theoretic Hecke characters, see \cite{Neukirch}, VII.10. Each factor is covered by Scholium (I), so $\zeta_{E_\ell}$ is zero-free in the stated half-plane. The extension $E_\ell/H_\ell$ is also abelian. Indeed, restriction gives an isomorphism
$$\Gal(E_\ell/H_\ell)\cong \Gal(F_\ell/(F_\ell\cap H_\ell)),$$
a subgroup of the abelian group $\Gal(F_\ell/\Q)$. We have
$$ \zeta_{E_\ell}(s)=\zeta_{H_\ell}(s) \prod_{\substack{\eta\in\widehat{\Gal(E_\ell/H_\ell)}\\\eta\neq\triv}}L(s,\eta).$$
The nonprincipal Hecke factors in the second product are entire. Hence, a zero of $\zeta_{H_\ell}(s)$ in the indicated region would therefore be a zero of $\zeta_{E_\ell}(s)$, a contradiction.
\end{proof}

The next result is another technical key-input, which may be found in a weaker form as Lem.\ 9.2 in \cite{OAI}. Its suggested proof, however, appears very sketchy and is lacking too many details, for us to be sure that the result's insinuated constant is {\it independent of the given number field}, whence we prefer to give a full proof of our more general result ourselves and to complete the arguments {\it ibidem} by ideas, which we developed independently in the course of the research leading to the author's recent results on the prime counting function, cf.\ \cite{grobner}.\\\\ 
To formulate the result, let $F$ be any number field, let $\varphi\in C_c^\infty(\R_{>0},\R_{\geq 0})$ be any non-negative, real valued smooth function of compact support on $\R_{>0}=(0,\infty)$. Its Mellin transform $\widehat \varphi(s)=\int_0^\infty \varphi(t)t^{s-1} dt$ is then an entire function. Now, for $x>0$ we define as in \cite{OAI} 
\begin{equation}\label{eq:thetadef}
\Theta_{F,\varphi}(x):=\sum_{\mathfrak q}\sum_{j\geq1}\log(N(\mathfrak q))\,\varphi\!\left(\frac{N(\mathfrak q)^j}{x}\right),
\end{equation}
where the first sum runs over all non-zero prime ideals $\q$ of $\O_F$. We remark that this sum is well-defined as it is in fact finite for each fixed $x$, because $\varphi$ has compact support and only finitely many prime-ideal powers have bounded norm. Diverting from {\it loc. cit.}, however, but following the spirit of \cite{grobner}, we set
$$\Theta_F:=\sup\{\Re e(\rho) \ | \ \zeta_F(\rho)=0, \ 0<\Re e(\rho)<1\}.$$
Then, $1/2\leq\Theta_F\leq 1$ and we may prove the following generalization, as well as refinement of Lem.\ 9.2 in \cite{OAI}:

\begin{prop}\label{lem:smooth} 
Let $F$ be any number field and let $\varphi\in C_c^\infty(\R_{>0},\R_{\geq 0})$ be any non-negative, real valued smooth function of compact support. Then, for $x>1$,
\begin{equation}\label{eq:theta}
 \Theta_{F,\varphi}(x) \ll_\varphi x\widehat\varphi(1)+x^{\Theta_F}(\log D_F+[F:\Q]),
\end{equation}
where the choice of subscript shall indicate that the implied constant is independent of $F$ and $x$. 
\end{prop}

\begin{proof}
Throughout this proof $x>1$ is an arbitrary, but fixed real number. Let $\varphi$ be as in the statement of the proposition. Repeated integration by parts gives for every fixed $B>0$
\begin{equation}\label{eq:individual}
|\widehat \varphi(\sigma+it)|\ll_{\varphi,B}(1+|t|)^{-B}
\end{equation}
uniformly on $-\frac12\leq\sigma\leq2$. As a consequence, Mellin inversion on $\Re e(s)=2$ yields
\begin{equation}\label{eq:Theta}
 \Theta_{F,\varphi}(x)=\frac1{2\pi i} \lim_{T\to\infty}\int^{2+iT}_{2-iT} -\frac{\zeta_F'(s)}{\zeta_F(s)}x^s\widehat \varphi(s)ds.
\end{equation}
Indeed, logarithmic differentiation of the Euler product for the Dedekind zeta function gives the usual Dirichlet series
\begin{equation}\label{eq:individual2}
-\frac{\zeta_F'(s)}{\zeta_F(s)}=\sum_{\mathfrak q}\sum_{j\geq1}\frac{\log(N(\mathfrak q))}{N(\mathfrak q)^{js}}\qquad (\Re e(s)>1),
\end{equation}
which is absolutely convergent in the given domain. Substituting this into the above integral and interchanging summation and integration, we obtain
\begin{align*}
&\frac{1}{2\pi i} \lim_{T\to\infty}\int^{2+iT}_{2-iT} -\frac{\zeta_F'(s)}{\zeta_F(s)} x^s\widehat{\varphi}(s) ds \\
&\qquad= \sum_{\mathfrak q}\sum_{j\geq1}\log(N(\mathfrak q))\frac{1}{2\pi i} \lim_{T\to\infty}\int^{2+iT}_{2-iT} \widehat{\varphi}(s)\left(\frac{x}{N(\mathfrak q)^j}\right)^s ds \\
&\qquad=\sum_{\mathfrak q}\sum_{j\geq1}\log(N(\mathfrak q)) \varphi\!\left(\frac{N(\mathfrak q)^j}{x}\right)
=\Theta_{F,\varphi}(x),
\end{align*}
where the second last equality is just the Mellin inversion formula for $y=N(\mathfrak q)^j/x$ on $\Re e(s)=2$, i.e.,
$$\varphi\!\left(\frac{N(\mathfrak q)^j}{x}\right)=\frac{1}{2\pi i} \lim_{T\to\infty}\int^{2+iT}_{2-iT} \widehat{\varphi}(s)\left(\frac{x}{N(\mathfrak q)^j}\right)^s ds.$$
For completeness, let us also note that the performed interchange is justified by absolute integrability, which itself follows from \eqref{eq:individual}:
$$
\sum_{\mathfrak q}\sum_{j\geq1}\int_{\mathbb R}\left|\log(N(\mathfrak q)) \widehat{\varphi}(2+it)\left(\frac{x}{N(\mathfrak q)^j}\right)^{2+it}\right|dt =
x^2\left(\sum_{\mathfrak q}\sum_{j\geq1}\frac{\log(N(\mathfrak q))}{N(\mathfrak q)^{2j}}\right)\int_{\mathbb R}|\widehat{\varphi}(2+it)| dt<\infty.
$$
We now want to shift the integral to $\Re e(s)=-1/2$, while keeping track of its dependence on $F$.\\\\ To this end, first, let
$$N_F(T):=\sum_{\substack{\zeta_F(\rho)=0\\ 0<\Re e(\rho)<1\\ |\Im m(\rho)|\leq T}}m_\rho$$
with $m_\rho:=\operatorname{ord}_{s=\rho}\zeta_F(s)$, be the (finite) number of the non-trivial zeros $\rho=\beta+i\gamma$ of the Dedekind zeta-function $\zeta_F(s)$ with $|\gamma|\leq T$, counted with multiplicity. We abbreviate $n_F:=[F:\Q]$ and use the uniform estimate provided in \cite{HSW}, Cor.\ 1.2 and (1.3) on p.\ 278, to obtain
\begin{equation}\label{eq:zeroasymptotic}
N_F(T)=M_F(T)+O\bigl(\log D_F+n_F\log(T+2)\bigr)
\qquad(T\geq1),
\end{equation}
with
$$M_F(T):=\frac{T}{\pi}\log\left(D_F\left(\frac{T}{2\pi e}\right)^{n_F}\right).$$
It is important to note that the implied constant is absolute, i.e., independent of $F$. In particular,
\begin{equation}\label{eq:zerocount}
N_F(T)\ll(T+1)\bigl(\log D_F+n_F\log(T+3)\bigr)
\qquad(T\geq0),
\end{equation}
with an absolute implied constant (For $0\leq T<1$, use $N_F(T)\leq N_F(1)$.) Splitting the non-trivial zeros of $\zeta_F(s)$ into $|\gamma|\leq1$ and the ranges $2^k<|\gamma|\leq2^{k+1}$ and ordering them by increasing absolute value of their imaginary parts $\gamma=\Im m(\rho)$ gives
\begin{align}
\sum_\rho m_\rho (1+|\gamma|)^{-3} &\leq N_F(1)+\sum_{k\geq0}2^{-3k}N_F(2^{k+1})\notag\\
&\ll\log D_F+n_F+\sum_{k\geq0}2^{-2k}\log D_F+\sum_{k\geq0}2^{-2k}n_F(k+1)\notag\\
&\ll\log D_F+n_F,
\label{eq:zerosum}
\end{align}
the last line being a consequence of the convergence of the two series.\\\\
Next, the functional equation controls the integral on the new left-hand line. Write
$$\Gamma_F(s)=\Gamma_{\mathbb R}(s)^{r_1}\Gamma_{\mathbb C}(s)^{r_2},\qquad n_F=r_1+2r_2,$$
where $\Gamma_{\mathbb R}(s)=\pi^{-s/2}\Gamma(s/2)$ and $\Gamma_{\mathbb C}(s)=2(2\pi)^{-s}\Gamma(s)$. Logarithmic differentiation of the functional equation for
$D_F^{s/2}\Gamma_F(s)\zeta_F(s)$ gives
$$\frac{\zeta_F'(s)}{\zeta_F(s)}=-\frac{\zeta_F'(1-s)}{\zeta_F(1-s)}-\log D_F-\frac{\Gamma_F'(s)}{\Gamma_F(s)}-\frac{\Gamma_F'(1-s)}{\Gamma_F(1-s)}.$$
On $\Re e(s)=-1/2$, using \eqref{eq:individual2}, we get
$$\left|\frac{\zeta_F'(1-s)}{\zeta_F(1-s)}\right|\leq\sum_p\sum_{\mathfrak q\mid p}\frac{f_{\mathfrak q}\log p}{p^{3f_{\mathfrak q}/2}-1}\leq n_F\sum_p\frac{\log p}{p^{3/2}-1}\ll n_F,$$
where $f_{\mathfrak q}$ denotes the residue degree (whence $\sum_{\mathfrak q\mid p}f_{\mathfrak q}\leq n_F$). We observe that the gamma factors all stay away from their poles on the relevant lines. So, for every $t\in\R$, Stirling's formula, together with boundedness on compact segments, gives
\begin{equation}\label{eq:leftline}
\left|\frac{\zeta_F'(-\frac12+it)}{\zeta_F(-\frac12+it)}\right| \ll\log D_F+n_F\log(2+|t|),
\end{equation}
with an absolute constant. Combining this with %$|x^{-1/2+it}|=x^{-1/2}$ 
\eqref{eq:individual} and the fact that $\widehat\varphi(s)$ is of rapid decay, we hence obtain 
\begin{equation}\label{eq:leftintegral}
\frac{1}{2\pi i}\lim_{T\to\infty}\int^{-1/2+iT}_{-1/2-iT}  -\frac{\zeta_F'(s)}{\zeta_F(s)}x^s\widehat \varphi(s) ds \ll_{\varphi} x^{-1/2}(\log D_F+n_F).
\end{equation}
Now, to justify the announced contour shift from $\Re e(s)=2$ to the line $\Re e(s)=-1/2$, we first fix $F$. For an integer $j\geq 2$ we let $M_j$ denote the number of non-trivial zeros of $\zeta_F(s)$ with $j-1\leq\Im m(\rho)\leq j+2$, counted with (finite) multiplicity $m_\rho$. For integers $j\geq4$, taking differences in
\eqref{eq:zeroasymptotic}, and using
$$M_F'(u)=\frac{1}{\pi}\left(\log D_F+n_F\log\frac{u}{2\pi}\right),$$
gives
\begin{equation}\label{eq:localzerocount}
M_j\leq N_F(j+2)-N_F(j-2)\ll \log D_F+n_F\log(j+4)\ll_F\log j.
\end{equation}
So we may choose $C_F\geq1$, depending only on $F$, such that $M_j\leq C_F\log j$ for all integers $j\geq4$, and we put $\epsilon_j=(4C_F\log j)^{-1}$. Removing intervals of radius $\epsilon_j$ around the ordinates $\gamma=\Im m(\rho)$ of the non-trivial zeros of $\zeta_F(s)$ in $j-1\leq\gamma\leq j+2$ excludes length at most $2\epsilon_j C_F\log j=1/2$ from $[j,j+1]$. Thus, we can choose $T_j\in[j,j+1]$ whose distance from every zero ordinate is at least $\epsilon_j$. As ordinates outside the specified window have distance at least $1$, the conjugation-symmetry of the non-trivial zeros of $\zeta_F(s)$ finally allows us to choose and fix for each $j\geq 4$ a $T_j\in[j,j+1]$ such that
\begin{equation}\label{eq:heightseparation}
|\pm T_j-\Im m(\rho)|\geq\epsilon_j
\end{equation}
for every non-trivial zero $\rho$.\\\\
We recall that for $s=\sigma+it$, with $|t|\geq 2$, and away from the zeros of $\zeta_F(s)$,
$$\frac{\zeta_F'(s)}{\zeta_F(s)}=-\frac1s-\frac1{s-1}+\sum_{\substack{\rho\\|\Im m(\rho)-t|\leq1}} m_\rho\left(\frac1\rho+\frac1{s-\rho}\right)+O_F(\log(|t|+4)),$$
uniformly for $-1/2\leq\sigma\leq2$. See \cite{HKMS}, Lem.\ 5.2. At $t=\pm T_j$, $j\geq 4$, the sum has at most $C_F\log j$ terms by the definition of $C_F$. Each such summand then satisfies $|s-\rho|^{-1}\leq4C_F\log(j)$ and $|\rho|^{-1}\leq(j-1)^{-1}$. Consequently, we have shown that for every $j\ge 4$, 
\begin{equation}\label{eq:horizontalbound}
\sup_{-1/2\leq\sigma\leq2}
\left|\frac{\zeta_F'(\sigma\pm iT_j)}{\zeta_F(\sigma\pm iT_j)}\right|\ll_F(\log T_j)^2.
\end{equation}
Now we compute the integral of $-\frac{\zeta_F'(s)}{\zeta_F(s)}x^s\widehat\varphi(s)$ along the rectangle with vertices $2\pm iT_j$ and $-1/2\pm iT_j$. On either horizontal side, \eqref{eq:horizontalbound} and
\eqref{eq:individual} with $B=2$, imply
$$\left|\int_{-1/2\pm iT_j}^{2\pm iT_j}-\frac{\zeta_F'(s)}{\zeta_F(s)}x^s\widehat\varphi(s) ds \right| \ll_{F,\varphi}x^2\frac{(\log T_j)^2}{T_j^2}\longrightarrow 0$$
as $j\to\infty$ (whence $T_j\to\infty$ as $T_j\in[j,j+1]$).
Indeed, either horizontal side has fixed length $5/2$, $|x^{\sigma\pm iT_j}|=x^\sigma\leq x^2$ for $-1/2\leq\sigma\leq2$, and $|\widehat\varphi(\sigma\pm iT_j)|\ll_\varphi T_j^{-2}$. Finally, we observe that the integrals on the vertical lines converge absolutely as $j\to\infty$ (i.e., $T_j\to\infty$): On $\Re e(s)=2$, this limit, oriented upwards and multiplied by $1/(2\pi i)$, is nothing else but $\Theta_{F,\varphi}(x)$ by \eqref{eq:Theta}, whereas on $\Re e(s)=-1/2$, the respective limit is given by \eqref{eq:leftintegral}.\\\\ 
Knowing the convergence and value of the respective integral along the rectangle with vertices $2\pm iT_j$ and $-1/2\pm iT_j$ as $j\to\infty$, we now compute its value in a second way by use of the residue theorem: The pole of $\zeta_F(s)$ at $s=1$ contributes $x\widehat \varphi(1)$, because $-\zeta_F'(s)/\zeta_F(s)$ has residue $+1$ there. A non-trivial zero $\rho$ of multiplicity $m_\rho$ of $\zeta_F(s)$ contributes $-m_\rho x^\rho\widehat \varphi(\rho)$. At $s=0$, the order of $\zeta_F$ is $m_0=r_1+r_2-1$, giving the contribution $-m_0\widehat \varphi(0)$. Indeed, the completed zeta function has a simple pole at $s=0$, whereas $\Gamma_F(s)$ has a pole of order $r_1+r_2$ there, whence $\textrm{ord}_{s=0}\zeta_F(s)=r_1+r_2-1$. As by \cite{IK04}, Thm.\ 5.33 and \cite{Neukirch}, Cor.\ VII.5.11.(iii), there are no zeros on $\Re e(s)=1$ or on $\Re e(s)=0$ (apart from the possible zero at $s=0$), the residue theorem and \eqref{eq:leftintegral} now give for $x>1$
\begin{equation}\label{eq:explicit}
\Theta_{F,\varphi}(x)= x\widehat \varphi(1)-\sum_\rho m_\rho x^\rho\widehat \varphi(\rho)-(r_1+r_2-1)\widehat \varphi(0)+O_{\varphi}\bigl(x^{-1/2}(\log D_F+n_F)\bigr),
\end{equation}
where the sum is over the non-trivial zeros $\rho$ of $\zeta_F(s)$ ordered by increasing absolute value of their imaginary parts. This latter sum is in fact absolutely convergent (hence their just mentioned ordering is irrelevant), invoking \eqref{eq:individual} with $B=3$ and \eqref{eq:zerosum}. In order to finish the proof of the proposition, it now suffices to bound the summands on the right hand side of \eqref{eq:explicit}: To this end, we apply again \eqref{eq:individual} with $B=3$ and \eqref{eq:zerosum} to get on the one hand,
$$\left|\sum_\rho m_\rho x^\rho\widehat \varphi(\rho)\right|\ll_{\varphi} x^{\Theta_F}\sum_\rho m_\rho (1+|\gamma|)^{-3}\ll_{\varphi} x^{\Theta_F}(\log D_F+n_F),$$
as clearly $|x^\rho|=x^{\Re e(\rho)}\leq x^{\Theta_F}$. On the other hand, the contribution at $0$ is $O_{\varphi}(n_F)$ and hence may be, together with the remainder in \eqref{eq:explicit}, absorbed into $O_{\varphi}(x^{\Theta_F}(\log D_F+n_F))$. This proves \eqref{eq:theta}.
\end{proof}

Let now $S_\ell(x)$ be the number of prime ideals $\p$ of $\O_K$, whose norm $N(\p)$ is a rational prime, satisfying $x< N(\p)<2x$ and that split completely in $H_\ell/K$. Recall $\delta_0=10^{-6}$, so $\Theta_{H_\ell}\leq 1-\delta_0$ for all odd primes $\ell$ by Lem.\ \ref{lem:descent}. Then we obtain the following

\begin{cor}\label{prop:splitting}
For all sufficiently large $x$ and every odd prime $\ell\nmid D$ with $\ell\leq \exp(\log(x)^{3/10})$,
\begin{equation}\label{eq:Sbound}
 S_\ell(x)\ll_K \frac{x}{\ell(\ell-1)\log(x)}+x^{1-\delta_0}.
\end{equation}
Both the implied constant and the threshold for $x$ are independent of $\ell$.
\end{cor}
\begin{proof}
Choose a function $\psi\in C_c^\infty((1/2,3),\R_{\geq 0})$, which is $\psi\geq1$ on $[1,2]$. Recall that $[H_\ell:K]=\ell(\ell-1)$ by Lem.\ \ref{lem:degree} . Hence, for every prime ideal counted by $S_\ell(x)$ there are exactly $\ell(\ell-1)$ primes of $H_\ell$ above it, all with the same norm. Their terms in \eqref{eq:thetadef}, attached to $j=1$, contribute at least a summand $\ell(\ell-1)\log(x)$. Positivity therefore gives
$$\ell(\ell-1)\log(x)S_\ell(x)\leq\Theta_{H_\ell,\psi}(x).$$
By Lem.\ \ref{lem:descent} and Prop.\ \ref{lem:smooth}, the right side is a $O_\psi( x+x^{1-\delta_0} \bigl(\log D_{H_\ell}+[H_\ell:\Q]\bigr)).$ Since $[H_\ell:\Q]=n\ell(\ell-1)$, Lem.\ \ref{lem:disc} bounds the second term after division by $\ell(\ell-1)\log(x)$ by a $O_K(x^{1-\delta_0}\frac{1+\log\ell}{\log(x)})$. For $\ell\leq \exp(\log(x)^{3/10})$ this last ratio is bounded independently of $\ell$ and $x$, assuming that $x$ is sufficiently large. This proves the claim.
\end{proof}

\section{Prime ideals of non-surjective reduction}\label{sec:goodprimes}

\subsection{Basic facts on the reduction map $\O_K^*\longrightarrow(\O_K/\p)^*$}
We now want to apply Scholium (II) with $k=m$. To this end, and to simplify matters, set for $x>1$,
$$\mathcal P_m(x):=\left\{p\text{ prime}\left|
\begin{array}{l}
x<p<2x,\quad p\equiv-1\pmod{8m},\\
\exists \ r\in\mathbb Z_{\ge1},\ \exists \ q\text{ prime, s.t.}:\\
p-1=2rq,\quad q>x^{9/10},\\
\text{every prime divisor }\ell\mid r
\text{ satisfies }\exp(\log(x)^{1/10})\le\ell\le \exp(\log(x)^{3/10})
\end{array}\right.
\right\}.$$
Then, by Scholium (II), there are constants $a_m,x_m>0$ such that for all $x\geq x_m$
\begin{equation}\label{eq:reservoirsize}
 |\mathcal P_m(x)|\geq a_m\frac{x}{\log(x)^2}.
\end{equation}
For the remainder of this section, we suppose that $x\geq \max\{x_m,3\}$. For each $p\in\mathcal P_m(x)$, fix integers $r=r(p,x)$ and $q=q(p,x)$ as in the definition above.

\begin{lem}\label{lem:progression}
{\rm (i):} Every $p\in\mathcal P_m(x)$ splits completely in $K$ as $p\O_K=\p_1\cdots\p_n$, where $\p_1,\ldots,\p_n$ are distinct prime ideals satisfying $\O_K/\p_j\cong\mathbb F_p$, $j=1,\ldots,n$. Moreover, $v_2(p-1)=1$, i.e., $2$ divides the even number $p-1$ exactly once.\\ {\rm (ii):} For any prime ideal $\p$ above $p$, let $H_\p$ denote
the image of the natural reduction map $\O_K^*\longrightarrow(\O_K/\p)^*$. Then, the index
$$i(\p):=[(\O_K/\p)^*:H_\p]=\frac{p-1}{|H_\p|}$$
is odd.
\end{lem}

\begin{proof}
Let $p\in\mathcal P_m(x)$. Since $p\equiv-1\pmod m$, we have $p\nmid m$, and the arithmetic Frobenius of $p$ in $\Q(\zeta_m)/\Q$ sends $\zeta_m$ to $\zeta_m^{-1}$. Thus, it is complex conjugation. Since $K$ is totally real, its restriction to $K$ is therefore the identity. Consequently, $p$ is unramified and splits completely in $K$, giving the asserted $n$ distinct prime ideals of norm $p$. Since $p\equiv-1\pmod{8m}$ as by assumption $p\in\mathcal P_m(x)$, we get in particular $8\mid(p+1)$, whence $p\equiv7\pmod8$. Writing $p=8c+7$ with $c\in\Z_{\geq0}$, we see that $p-1=8c+6=2(4c+3)$ and as $4c+3$ is odd, $v_2(p-1)=1$.\\\\
Now fix any prime ideal $\p$ above $p$. The group $(\O_K/\p)^*\cong\F_p^*$ has order $p-1$. Since $-1\in\O_K^*$ and $p=8c+7$ is odd, $H_\p$ contains an element of order $2$ and is hence of even order itself, i.e., $|H_\p|$ is even. As $|H_\p|$ divides $p-1$ and $v_2(p-1)=1$, we obtain $v_2(|H_\p|)=1$ and $v_2(i(\p))=v_2(p-1)-v_2(|H_\p|)=0$. Therefore, $i(\p)$ is odd.
\end{proof}

\noindent Thus, each $p\in\mathcal P_m(x)$ provides $n$ distinct prime ideals of $\O_K$ of norm $p$. If $\p$ is any such prime ideal above $p$, and if a prime $\ell$ divides $i(\p)$, it is therefore odd and divides $p-1$, both by Lem.\ \ref{lem:progression}, and so the subgroup $H_\p$ is contained in the subgroup of $\ell$-th powers in the cyclic group $(\O_K/\p)^*\cong\mathbb F_p^*$. In particular,
$$ \varepsilon^{(p-1)/\ell}\equiv1\pmod\p.$$
If $i(\p)>1$, i.e., if the natural reduction map $\O_K^*\longrightarrow(\O_K/\p)^*$ is not surjective, Lem.\ \ref{lem:progression} thus gives such an odd prime divisor $\ell$ of $i(\p)$. Since $p-1=2rq$, it either divides $r$ or equals $q$. 

\subsection{Exceptional divisors of $i(\p)$: $\ell\mid r$}\label{sect:1} 

Let $A(x)$ denote the number of distinct prime ideals $\p$ above primes $p\in\mathcal P_m(x)$ for which some prime divisor $\ell$ of $r$ also divides $i(\p)$. For sufficiently large $x$, every such $\ell$ lies in the interval prescribed by Scholium (II), and every prime $\ell\in[\exp(\log(x)^{1/10}),\exp(\log(x)^{3/10})]$ is odd and does not divide $D$. Also, $\exp(\log(x)^{3/10})<x^{9/10}$, so $\ell\neq q$. If such an $\ell$ also divides $i(\p)$, then $\ell\mid r$ and Lem.\ \ref{lem:splittingtest} shows that $\p$ is counted by $S_\ell(x)$. By Cor.\ \ref{prop:splitting}, the number $A(x)$ is hence bounded above by 
\begin{align}
 A(x)\leq \sum_{\exp(\log(x)^{1/10})\leq\ell\leq \exp(\log(x)^{3/10})}S_\ell(x)
 &\ll_K\frac{x}{\log(x)}
 \sum_{n\geq\lceil \exp(\log(x)^{1/10})\rceil}\frac1{n(n-1)}
 +\exp(\log(x)^{3/10})x^{1-\delta_0}\notag\\
 &\ll_K\frac{x}{\exp(\log(x)^{1/10})\log(x)}+\exp(\log(x)^{3/10})x^{1-\delta_0}.
 \label{eq:smallexceptions}
\end{align}
Here, obviously the first sum on the left is over primes $\ell$.\\\\
We observe that both terms in the last line of \eqref{eq:smallexceptions} are in fact an $o(x/\log(x)^2)$ as $x\to\infty$. Indeed, dividing by $x/\log(x)^2$ and putting $t=\log(x)$, $u=\log(x)^{1/10}$, then their respective ratios satisfy
$$u^{10}e^{-u}\longrightarrow0,$$ and
$$t^2e^{t^{3/10}-\delta_0t}\leq t^2e^{-\delta_0t/2}\longrightarrow 0$$
as $x\to\infty$. The latter inequality holds for all sufficiently large $t$, since $t^{3/10}/t=t^{-7/10}\to0$ and our $\delta_0>0$ -- whose existence is guaranteed by Scholium (I) --  is fixed and strictly positive. So, in summary,
\begin{equation}\label{eq:smallorderA}
A(x)=o\left(\frac{x}{\log(x)^2}\right)
\end{equation}

\subsection{Exceptional divisors of $i(\p)$: $\ell = q$}\label{sect:2}

Suppose that $q\mid i(\p)$. The image $H_\p$ of the unit group of $\O_K$ in $(\O_K/\p)^*$ has order $(p-1)/i(\p)$. Consequently, the order of $\varepsilon\bmod\p$ divides $(p-1)/i(\p)$, which in turn divides $(p-1)/q$. Since $p<2x$ and $q>x^{9/10}$, we obtain
\begin{equation}\label{eq:smallorder}
\operatorname{ord}_{\p}(\varepsilon)
\leq\frac{p-1}{q}<2x^{1/10}.
\end{equation}
Abreviate $J:=\lfloor2x^{1/10}\rfloor$ and
$$\alpha_x:=\prod_{j=1}^{J}(\varepsilon^j-1)\in\O_K\setminus\{0\}.$$
The product is non-zero because $\varepsilon>1$ (and $x\geq 3$). Since $\varepsilon^{{\rm ord}_{\p}(\varepsilon)}-1\in\p$, \eqref{eq:smallorder} implies that $\alpha_x\in\p$,
or equivalently $\p\mid(\alpha_x)$. Let $\sigma$ run over the $n$ real embeddings $K\hra\R$, and set
$$C_\varepsilon:=\sum_{\sigma:K\hra\R}\log(\max\{1,|\sigma(\varepsilon)|\}).$$
It follows that
$$\bigl|N_{K/\Q}(\varepsilon^j-1)\bigr|=\prod_\sigma|\sigma(\varepsilon)^j-1|\leq2^n\prod_\sigma\max\{1,|\sigma(\varepsilon)|\}^j=2^n\exp(jC_\varepsilon).$$
We get
$$N((\alpha_x))=\bigl|N_{K/\Q}(\alpha_x)\bigr|\leq2^{nJ}\exp\left(C_\varepsilon\frac{J(J+1)}2\right).$$
Taking logarithms gives
\begin{equation}\label{eq:normproduct}
\log N((\alpha_x))\leq nJ\log2+C_\varepsilon\frac{J(J+1)}2
\ll_K J^2\ll_K x^{1/5}.
\end{equation}
Write the prime ideal factorization as $(\alpha_x)=\prod_{\mathfrak q}\mathfrak q^{e_{\mathfrak q}}$. Then,
$$\log N((\alpha_x))=\sum_{\mathfrak q}e_{\mathfrak q}\log N(\mathfrak q).$$
Each prime ideal $\p$ with $q\mid i(\p)$ occurs with $e_{\p}\geq1$ and has $N(\p)>x$. Thus, if $B(x)$ denotes the number of distinct such prime ideals $\p$ above primes $p\in\mathcal P_m(x)$, then $\log N((\alpha_x))=\sum_{\mathfrak q}e_{\mathfrak q}\log N(\mathfrak q)\geq B(x)\log(x)$. Together with \eqref{eq:normproduct}, this yields
$$B(x)\ll_K\frac{x^{1/5}}{\log(x)}.$$
So, dividing again by $x/\log(x)^2$ and observing that $\frac{x^{1/5}/\log(x)}{x/\log(x)^2}=\frac{\log(x)}{x^{4/5}}\longrightarrow 0$, as $x\to\infty$, we obtain again 
\begin{equation}\label{eq:largeexceptions}
B(x)=o\left(\frac{x}{\log(x)^2}\right).
\end{equation}

\section{Proof of the Main Theorem}\label{sect:mthmproof}
\subsection{} 
We are finally ready to give the proof of our main result. So, $K/\Q$ is totally real and abelian, with $n=[K:\Q]\geq2$, and $m$ is the fixed cyclotomic modulus chosen above.

\begin{proof}[Proof of the Main Theorem:]
Let $x_m,a_m>0$ be the constants provided by Scholium (II) with $k=m$ and set $c_K:=(n-1)a_m$. By Lem.\ \ref{lem:progression} the rational primes in $\mathcal P_m(x)$ provide exactly $n|\mathcal P_m(x)|$ distinct prime ideals $\p$ of $\O_K$, whose norm is a rational prime. Those, for which $i(\p)>1$ belong to at least one of the two exceptional sets just bounded in cardinality in Sections \ref{sect:1} and \ref{sect:2}, i.e., more precisely, the number of these prime ideals $\p$ is bounded above by $A(x)+B(x)$. Therefore, recalling \eqref{eq:reservoirsize}, \eqref{eq:smallorderA} and \eqref{eq:largeexceptions} and observing that surjectivity of the reduction map $\O_K^*\ra (\O_K/\p)^*$ is equivalent to $i(\p)=1$, we obtain
$$
\begin{aligned}
\Pi_K(x)
&\geq n|\mathcal P_m(x)|-A(x)-B(x)\\
&\geq na_m\frac{x}{\log(x)^2}
      -o\!\left(\frac{x}{\log(x)^2}\right).
\end{aligned}
$$
for $x\gg 1$ and we may hence choose an $x_K\geq\max\{x_m,3\}$ large enough such that all preceding estimates apply and that $A(x)+B(x)\leq a_m\frac{x}{\log(x)^2}$ for $x\geq x_K$. This establishes the claimed lower bound of our Main Theorem.\\\\ 
The asserted claim about $\O_K$ being Euclidean, if $h_K=1$, follows from this, as it is a direct application of \cite{HM}, Prop.\ 2. Indeed, for sufficiently large $y$, the number of prime ideals $\p$ with $N(\p)\leq y$ and for which $\O_K^*\ra (\O_K/\p)^*$ is surjective, is at least $\Pi_K(y/2)\gg_K y/(\log y)^2$. When $h_K=1$, these ideals correspond to prime elements up to association, so finally Prop.\ 2 of \cite{HM}, with the empty admissible set, applies. The converse assertion and the remaining case of $[K:\Q]=1$ are both trivially covered.
\end{proof}

\section{Complementa}\label{sec:extensions}

\subsection{Extensions beyond totally real abelian fields}
We conclude by outlining an extension of the argument to number fields which need not be totally real, nor abelian. In this section, $K$ denotes an arbitrary number field of positive unit rank $r_K=r_1+r_2-1$, where $(r_1,r_2)$ is its signature, $r_1$ denoting the number of real places and $r_2$ the number of complex places of $K$. We write $\mu(K)$ for its group of roots of unity and use $\Pi_K(x)$ with the same definition as in the Main Theorem.\\\\ 
The additional analytic input, which we want to allow ourselves to use, is \cite{Liu26}, Thm.\ 17.1: It announces the existence of a constant $B_\flat<1$, such that for every number field $F$ the Dedekind zeta function $\zeta_F(s)$ has no zero in $\Re e(s)>B_\flat$. In this supplementary section, we use this result, like the inputs in the Introduction, as a blackbox, and put $\delta_1=1-B_\flat>0$.\\\\ 
We shall also use \cite{OAI}, Prop.\ 2.1, in the case $c=2$ and $M=8k$: For any fixed, coprime integers $k\geq 1$ and $u$ satisfying $\gcd(u-1,8k)=2$ and every sufficiently large $x$, it supplies $\gg_{k,u}x/\log(x)^2$ rational primes $x<p<2x$ with $p\equiv u\pmod{8k}$ and $p-1=2rq$, where $q>x^{9/10}$ is prime and $r$ is a positive integer whose prime divisors $\ell$ all satisfy $\exp(\log(x)^{1/10})\leq\ell\leq\exp(\log(x)^{3/10})$.

\subsection{A criterion for arithmetic progressions}

\begin{prop}\label{prop:ext-progression}
Let $K$ be a number field with $r_K>0$. Suppose that there are coprime integers $k\geq 1$ and $u$, satisfying $\gcd(u-1,8k)=2$, such that every sufficiently large rational prime $p\equiv u\pmod{8k}$ has at least one prime ideal $\p$ of $K$ with $N(\p)=p$. Using the two analytic inputs just specified, there are constants $c_K,x_K>0$ such that
$$\Pi_K(x)\geq c_K\frac{x}{\log(x)^2}\qquad(x\geq x_K).$$
In particular, $\O_K$ is Euclidean if and only if $h_K=1$.
\end{prop}

\begin{proof}
Fix $k,u$ as above and choose a unit $\varepsilon$ whose class belongs to a basis of the free abelian group $\O_K^*/\mu(K)$. All these choices are independent of $x$ and $\ell$. For a prime ideal $\p$ of norm $p$, retain the notation
$$i(\p):=[(\O_K/\p)^*:\operatorname{im}(\O_K^*\ra(\O_K/\p)^*)].$$ For each of the primes supplied by \cite{OAI}, Prop.\ 2.1, choose one prime ideal $\p$ of norm $p$. Since $p\equiv3\pmod4$, the image of $-1$ and $v_2(p-1)=1$ imply that $i(\p)$ is odd. Hence every prime divisor of $i(\p)>1$ divides $r$ or equals $q$.\\\\
For every odd prime $\ell\nmid D_K$, the proofs of Lemmas~\ref{lem:degree}, \ref{lem:disc} and \ref{lem:splittingtest} apply to $B_\ell=K(\zeta_\ell)$ and $H_\ell=K(\zeta_\ell,\varepsilon^{1/\ell})$ without assuming total reality or abelianness: The chosen unit is not an $\ell$-th power in $K$, and the same norm argument excludes an $\ell$-th root in $B_\ell$. Hence, 
$$[H_\ell:K]=\ell(\ell-1),\qquad
\frac{\log D_{H_\ell}}{[H_\ell:\Q]}\ll_K1+\log\ell.$$
Let $S_\ell(x)$ again denote the number of prime ideals $\p$ of $K$ whose norm is a rational prime in $(x,2x)$ and which split completely in $H_\ell/K$. By \cite{Liu26}, Thm.\ 17.1, we have $\Theta_{H_\ell}\leq B_\flat$. Applying Prop.\ \ref{lem:smooth} and the positivity argument of Cor.\ \ref{prop:splitting} therefore gives
\begin{equation}\label{eq:ext-splitting}
S_\ell(x)\ll_K\frac{x}{\ell(\ell-1)\log(x)}+x^{1-\delta_1}
\end{equation}
uniformly for the sufficiently large primes $\ell\leq\exp(\log(x)^{3/10})$. For sufficiently large $x$, we have $p\nmid D_K$, and every prime divisor $\ell$ of $r$ is odd, does not divide $D_K$, and satisfies $\ell<p$. If such an $\ell$ divides $i(\p)$, the image of the unit group is contained in the subgroup of $\ell$-th powers in $(\O_K/\p)^*$. In particular, $\ell\mid p-1$ and $\varepsilon^{(p-1)/\ell}\equiv1\pmod\p$. Lem.\ \ref{lem:splittingtest} therefore shows that $\p$ is counted by $S_\ell(x)$. Consequently, the union bound of Sect.\ \ref{sect:1} controls all index divisors coming from $r$.\\\\
For the large divisor $q$, the same order bound gives $\operatorname{ord}_{\p}(\varepsilon)<2x^{1/10}$. We observe that the norm estimate remains valid for complex embeddings as well. Indeed, with $J=\lfloor2x^{1/10}\rfloor$ and $\alpha_J=\prod_{j=1}^J(\varepsilon^j-1)$, which is non-zero since $\varepsilon$ has infinite order, one has
$$\log N((\alpha_J))\leq [K:\Q]J\log2+ \frac{J(J+1)}2\sum_{\sigma:K\hra\C}\log\max\{1,|\sigma(\varepsilon)|\} \ll_K J^2,$$
where all embeddings are counted, including both members of each complex-conjugate pair. Therefore,  Sect.\ \ref{sect:2} applies unchanged and the number of exceptional chosen prime ideals is at most
$$\ll_K \frac{x}{\exp(\log(x)^{1/10})\log(x)} +\exp(\log(x)^{3/10})x^{1-\delta_1} +\frac{x^{1/5}}{\log(x)} =o\!\left(\frac{x}{\log(x)^2}\right).$$
Subtracting these exceptions from the chosen $\gg_Kx/\log(x)^2$ prime ideals proves the asserted bound. So, the conclusion of Euclideanity follows from \cite{HM}, Prop.\ 2, exactly as in the proof of the Main Theorem.
\end{proof}

\subsection{Applications in small unit rank}

Let $N$ be the normal closure of $K/\Q$, set $G=\Gal(N/\Q)$ and $H=\Gal(N/K)$, and write $G'=[G,G]$. The field $F=N^{G'}$ is abelian over $\Q$, so we may choose $m$ with $F\subseteq\Q(\zeta_m)$. A congruence modulo a multiple of $m$ hence specifies the Frobenius in $\Gal(F/\Q)=G/G'$. If the progression satisfies the coprimality conditions of Prop.\ \ref{prop:ext-progression} and every element of the specified coset of $G'$ fixes a point of $G/H$, that proposition applies: The Frobenius cycle lengths on $G/H$ are the residue degrees above an unramified prime, so a fixed point gives a prime ideal of norm $p$.\\\\
With this in mind we consider the following table: In the cubic cases, $L$ denotes the quadratic resolvent. The notation $D_4$, resp.\ $D_5$, means the dihedral group
of order $8$, resp.\ $10$, and the semidirect product of the cyclic groups of order $5$ and $4$, respectively, $C_5\rtimes C_4\cong \operatorname{AGL}_1(\mathbb F_5)$, is the group of affine maps $z\mapsto az+b$ with $a\in\mathbb F_5^*$ and $b\in\mathbb F_5$, of order $20$. In each row, {\it the signature in the second column is part of the specification} of the family of number fields, (whereas the assertion of its unit rank is a simple consequence).

\begin{table}[htbp]
\centering
\begingroup
\small
\setlength{\tabcolsep}{4pt}
\renewcommand{\arraystretch}{1.18}
\begin{tabular}{@{}
>{\raggedright\arraybackslash}p{.65\linewidth}
>{\centering\arraybackslash}p{.17\linewidth}
>{\centering\arraybackslash}p{\dimexpr.18\linewidth-4\tabcolsep\relax}
@{}}
\toprule
\textbf{number field $K$} & \textbf{signature} & \textbf{unit rank}\\
\midrule
$[K:\Q]=3$, non-Galois  
& $(1,1)$ & $1$\\
\addlinespace[.7em]
$[K:\Q]=3$, non-Galois, with $L\neq\Q(\sqrt3)$
& $(3,0)$ & $2$\\
\addlinespace[.7em]
$[K:\Q]=4$, with $G\cong D_4$
& $(2,1)$ & $2$\\
\addlinespace[.7em]
$[K:\Q]=4$, with $G\cong A_4$
& $(0,2)$ & $1$\\
\addlinespace[.7em]
$[K:\Q]=5$, with $G\cong D_5$ or $C_5\rtimes C_4$
& $(1,2)$ & $2$\\
\bottomrule
\end{tabular}
\endgroup
\caption{Applications in low unit-rank}
\label{tab:extensions}
\end{table}

\begin{cor}\label{cor:ext-families}
Each number field $K$ in Table~\ref{tab:extensions}, satisfies the assumptions of Prop.\ \ref{prop:ext-progression}. In particular, $h_K=1$ implies that $\O_K$ is Euclidean
throughout each of these families. 
\end{cor}

\begin{proof}
\noindent\underline{\it Cubic fields:}
Here $G=S_3$, $G'=A_3$, and $L=N^{A_3}$ is quadratic. Write $d_L$ for its fundamental discriminant and $\chi_L$ for the primitive quadratic Dirichlet character of conductor $|d_L|$. For all sufficiently large rational primes $p$, the extension $N/\Q$ is unramified at $p$. Now choose a Frobenius element $\sigma_p\in G=S_3$. Since $L=N^{A_3}$, restriction to $L$ identifies the quotient map $S_3\ra S_3/A_3$ with the sign map $S_3\ra\{\pm1\}$. Consequently, $\chi_L(p)=\operatorname{sgn}(\sigma_p)$ and so we see that, if $\chi_L(p)=-1$, then
$$
\sigma_p\in S_3\setminus A_3 =(12)A_3=\{(12),(13),(23)\}.
$$
Each of these transpositions exchanges two of the three embeddings of $K$ into $N$ and fixes the remaining one. Its cycle lengths are therefore $2$ and $1$. Since these cycle lengths are precisely the residue degrees of the prime ideals of $K$ above $p$, the splitting type of $p$ in $K$ is $(1,2)$. In particular, there is a prime ideal $\p$ of $K$ with $N(\p)=p$.\\\\
If $K$ has signature $(1,1)$, then $L$ is imaginary quadratic, so $\chi_L(-1)=-1$ holds and one may choose $k=|d_L|$ and $u=-1$.\\\\ 
If $K$ is totally real, then $d_L>0$. We set $k=d_L$ and choose an integer $u$ coprime to $k$ satisfying $\gcd(u-1,8k)=2$ and $\chi_L(u)=-1$. Such a choice exists whenever $L\neq\Q(\sqrt3)$. To see this, suppose first that an odd prime $v\geq5$ divides $d_L$. Write
$$
d_L=2^a d',
$$
where $a\in\{0,2,3\}$ and $d'$ is odd and squarefree. We have
$$
\chi_L(t)=\chi_2(t) \prod_{\substack{w\mid d_L\\ w\text{ odd prime}}} \left(\frac{t}{w}\right),
$$
where $\chi_2$ denotes its component at $2$, taken to be trivial if $a=0$, and the other factors are the respective Legendre symbols. Indeed, primitivity implies that the component at each odd prime $w\mid d_L$ is the non-trivial quadratic character modulo $w$. Since $2^a\mid8$, the value of $\chi_2(t)$ is determined by the residue of $t$ modulo $8$.\\\\
We first fix $u\equiv3\pmod8$ and $u\equiv-1\pmod w$ for every odd prime $w\mid d_L$ with $w\neq v$. At each such $w$, the prescribed residue is neither $0$ nor $1$, so $w$ divides neither $u$ nor $u-1$. Moreover, these prescriptions fix all local character values except the one at $v$. Hence, we can write
$$
\chi_L(u)=C\cdot\left(\frac{u}{v}\right),\qquad C:=\chi_2(3) \prod_{\substack{w\mid d_L\\ w\text{ odd prime}\\ w\neq v}} \left(\frac{-1}{w}\right)\in\{\pm1\},
$$
where an empty product is understood to be $1$. It remains to choose a residue modulo $v$ which is neither $0$ nor $1$ and makes the last Legendre symbol equal to $-C$.\\\\
Among the $v-1$ non-zero residues modulo $v$, exactly $(v-1)/2$ have Legendre symbol $1$ and exactly $(v-1)/2$ have Legendre symbol $-1$. The residue $1$ belongs to the first group. After excluding it, there remain $(v-3)/2\geq1$ residues with symbol $1$ and $(v-1)/2\geq2$ residues with symbol $-1$, since $v\geq5$. Consequently, we can choose an integer $b$ with
$$
2\leq b\leq v-1,
\qquad
\left(\frac{b}{v}\right)=-C.
$$
This choice ensures both the required local character value and the conditions $b\not\equiv0,1\pmod v$.\\\\
We now combine these choices by imposing the simultaneous set of conditions
$$
\begin{aligned}
u&\equiv3\pmod{2^{a+3}},\\
u&\equiv b\pmod v,\\
u&\equiv-1\pmod w
&&\text{for every odd prime }w\mid d_L,\ w\neq v.
\end{aligned}
$$
The moduli are pairwise coprime, so the Chinese remainder theorem gives a residue class modulo
$$
2^{a+3}d'=8d_L=8k.
$$
%The first congruence eaves the value $\chi_2(u)=\chi_2(3)$ unchanged.\\\\
Finally, $u$ is odd and is non-zero modulo every odd prime dividing $k=d_L$, so $\gcd(u,k)=1$. Likewise, no odd prime divisor of $8k$ divides $u-1$, because all prescribed residues for odd primes differ from $1$. Since $u-1\equiv2\pmod8$, %it is divisible by $2$, but not by $4$, whence 
we have $\gcd(u-1,8k)=2$. By construction, we also get
$$
\chi_L(u)=C\left(\frac{b}{v}\right)=-C^2=-1.
$$
as required.\\\\
If no odd prime $v\geq5$ divides $d_L$, the positive fundamental discriminant is $8$, $12$ or $24$. For $d_L=8$, use $u\equiv3\pmod8$, whereas for $d_L=24$, use $u\equiv11\pmod{24}$. Finally, the case $d_L=12$ is excluded from the table, as we supposed $L\neq\mathbb Q(\sqrt3)$, since in this case we would have $u\equiv11\pmod{12}$ and hence $\chi_L(u)=1$.\\\\

\noindent\underline{\it Quartic fields with $G\cong D_4$:}
Since $K$ has signature $(2,1)$, complex conjugation $j\in G$ fixes the two real embeddings of $K$ into $N$ and exchanges the two non-real embeddings. Its cycle type on the four embeddings is hence $(1,1,2)$. In the usual action of $D_4$ on the vertices of a square, $j$ is therefore a reflection fixing two opposite vertices. Choose a generator $a$ of the rotation subgroup, so that
$$
G=D_4=\langle a,j:a^4=j^2=1,\ jaj=a^{-1}\rangle.
$$
The commutator $aja^{-1}j^{-1}$ equals $a^2$, and the quotient by $\langle a^2\rangle$ is abelian. Consequently, $G'=\langle a^2\rangle$. Moreover, conjugation by $a^i$ sends $j$ to $a^{2i}j$, and every element of $G$ has the form $a^i$ or $a^ij$. It follows that the conjugacy class of $j$ equals $\{j,a^2j\}=jG'$. In particular, every element of the coset $jG'$ is conjugate to $j$ and has the same cycle type $(1,1,2)$.\\\\
Recall that $F\subseteq\Q(\zeta_m)$  and set $k=m$ and $u=-1$. Then, clearly, $\gcd(u,k)=\gcd(-1,m)=1$ and $\gcd(u-1,8k)=\gcd(-2,8m)=2$. We shall hence show that every sufficiently large rational prime in the progression $p\equiv u\pmod{8k}$ has a prime ideal of norm $p$ in $K$.\\\\
To this end, let $p\equiv-1\pmod{8m}$ be prime. For sufficiently large $p$, we have $p\nmid m$ and $p$ is unramified in $N$. In particular, $p\equiv-1\pmod m$, so the arithmetic Frobenius $\tau_p$ in $\Gal(\Q(\zeta_m)/\Q)$ acts by
$$
\tau_p(\zeta_m)=\zeta_m^p=\zeta_m^{-1}.
$$
Thus $\tau_p$ is complex conjugation, and its restriction to $F$ equals $j|_F$. Now choose a Frobenius element $\sigma_p\in G$ for $N/\Q$. By compatibility of Frobenius with restriction, and since $F/\Q$ is abelian, we obtain
$$
\sigma_p|_F=\tau_p|_F=j|_F.
$$
The restriction map $G\ra\Gal(F/\Q)$ has kernel $\Gal(N/F)=G'$. Hence the last equality implies $j^{-1}\sigma_p\in G'$, or, equivalently $\sigma_p\in jG'$. Hence, by what we have seen above, $\sigma_p$ is conjugate to $j$. Its cycle type on the embeddings of $K$ is therefore $(1,1,2)$. As the cycle lengths of $\sigma_p$ are the residue degrees of the prime ideals of $K$ above $p$, it factors as
$$
p\O_K=\p_1\p_2\q, \qquad N(\p_1)=N(\p_2)=p,\qquad N(\q)=p^2,
$$
with distinct prime ideals $\p_1,\p_2,\q$. In particular, we obtain prime ideals of $\O_K$ of norm $p$ as required.\\\\

\noindent\underline{\it Quartic fields with $G\cong A_4$:}
Let $V_4$ be the normal Klein four subgroup of $A_4$. The quotient $A_4/V_4$ is cyclic of order $3$, so $G'\subseteq V_4$. Since $A_4$ is non-abelian, $G'$ is non-trivial. Moreover, the three non-identity elements of $V_4$ form one conjugacy class in $A_4$, whence normality of $G'$ gives $G'=V_4$ and so $F=N^{V_4}$ is cyclic cubic over $\Q$. We choose a primitive Dirichlet character $\chi$ of order $3$ and conductor $m$ defining $F$, so that
$$
\ker\chi=\Gal(\Q(\zeta_m)/F)
$$
under the usual identification $\Gal(\Q(\zeta_m)/\Q)\cong(\Z/m\Z)^*$. Put $k=m$. We will now construct an integer $u$ satisfying $\gcd(u,k)=1$, $\gcd(u-1,8k)=2$ and $\chi(u)\neq1$.\\\\
Write $m=\prod_{v\mid m}v^{e_v}$ and $\chi=\prod_{v\mid m}\chi_v$ for its local decomposition, where $v$ runs over the prime divisors of $m$ and $\chi_v$ has conductor $v^{e_v}$. The conductor $m$ is odd, since the unit group modulo any power of $2$ is a $2$-group and hence admits no non-trivial character of order $3$. For $v\neq3$, an order-three character modulo $v^e$ is trivial on the kernel of reduction to $(\Z/v\Z)^*$, a $v$-group. It follows that a non-trivial primitive local component has conductor $v$, with $v\equiv1\pmod3$. In turn, at $v=3$, a non-trivial primitive cubic character has conductor $9$: The groups $(\Z/3^e\Z)^*$ are cyclic of order $2\cdot3^{e-1}$, and their order-three characters factor through $(\Z/9\Z)^*$.\\\\
Now choose a non-trivial local component $\chi_v$. If $v\neq3$, choose a unit $b$ modulo $v$ with $\chi_v(b)\neq1$. Then $b\not\equiv1\pmod v$. If $v=3$, choose $b=2$ modulo $9$. Since $2$ generates $(\Z/9\Z)^*$, we again have $\chi_v(b)\neq1$, and $b\not\equiv0,1\pmod3$. We impose $u\equiv b\pmod{v^{e_v}}$ for this selected prime $v$, and $u\equiv-1\pmod{w^{e_w}}$ for every other prime $w\mid m$. We also impose $u\equiv3\pmod8$. As these moduli are pairwise coprime, the Chinese remainder theorem gives a residue class modulo $8m$.\\\\ 
Every prescribed residue is neither $0$ nor $1$ modulo its underlying odd prime. Together with $u\equiv3\pmod8$, this gives $\gcd(u,k)=1$ and $\gcd(u-1,8k)=2$, since $k=m$. Furthermore, each cubic character satisfies $\chi_w(-1)=1$, since its value at $-1$ has order dividing both $2$ and $3$. Hence, we get
$$
\chi(u)=\chi_v(b)\prod_{\substack{w\mid m\\w\neq v}} \chi_w(-1)=\chi_v(b)\neq1.
$$
Now, let $p\equiv u\pmod{8k}$ be a sufficiently large rational prime, so that it is unramified in $N$. Then $\chi(p)=\chi(u)\neq1$, and a Frobenius element $\sigma_p\in G$ has non-trivial image in $G/V_4$. It therefore lies outside $V_4$. Every element of $A_4\setminus V_4$ is a $3$-cycle, so its cycle type on the four embeddings is $(1,3)$. The corresponding factorization hence provides
$$
p\O_K=\p\q,\qquad N(\p)=p,\qquad N(\q)=p^3,
$$
with distinct prime ideals $\p,\q$, as desired.\\\\

\noindent\underline{\it Quintic fields:}
In both cases, mentioned in our table, $G$ has a normal cyclic subgroup $T$ of order $5$, whose non-identity elements act as $5$-cycles on the five embeddings of $K$ into $N$. After labelling these embeddings by $\mathbb F_5$, we may therefore identify $T$ with the translation group
$$
T=\{z\mapsto z+b:b\in\mathbb F_5\}.
$$
Since every element of $G$ normalizes $T$, it acts as an affine map $z\mapsto az+b$. Indeed, if conjugation by $g\in G$ sends the translation $z\mapsto z+1$ to $z\mapsto z+a$, then $g(z+1)=g(z)+a$, and hence $g(z)=az+g(0)$. The possible multipliers form $\{\pm1\}$ in the case of $D_5$ and $\mathbb F_5^*$ in the case of 
$C_5\rtimes C_4=\operatorname{AGL}_1(\mathbb F_5)$.\\\\
The map sending an affine transformation to its multiplier is a homomorphism with kernel $T$ and abelian image. Consequently, $G'\subseteq T$. Both groups are nonabelian, so $G'$ is non-trivial. As $T$ has prime order $5$, this proves that $G'=T$. Hence, two elements of $G$ have the same image in $G/G'$ if and only if they have the same multiplier.\\\\
Since $K$ has signature $(1,2)$, complex conjugation $j\in G$ fixes the real embedding and exchanges the two pairs of non-real embeddings. In particular, $j$ is a non-trivial involution. Write $j(z)=az+b_0$. The identity $j^2=1$ gives
$$
a^2=1,\qquad (a+1)b_0=0.
$$
If $a=1$, the second equality would imply $b_0=0$, making $j$ the identity. Therefore $a=-1$, and the coset of $j$ modulo $G'$ is $jG'=\{z\mapsto-z+b:b\in\mathbb F_5\}$. Every element of this coset is an involution with exactly one fixed point: Its fixed-point equation is $z=-z+b$, i.e., $2z=b$, which has a unique solution in $\mathbb F_5$. Its cycle type on the five embeddings is therefore
$(1,2,2)$.\\\\
Now recall again our positive integer $m$ with $F\subseteq\Q(\zeta_m)$ and set $k=m$ and $u=-1$. As in the $D_4$-quartic case, for every sufficiently large rational prime $p\equiv u\pmod{8k}$, the congruence $p\equiv-1\pmod m$ gives $\sigma_p|_F=j|_F$ for a Frobenius element $\sigma_p\in G$. Hence $\sigma_p\in jG'$ and its cycle type is $(1,2,2)$. Since its cycle lengths are the residue degrees of the prime ideals of $K$ above $p$, we obtain
$$
p\O_K=\p\q_1\q_2, \qquad N(\p)=p,\qquad N(\q_1)=N(\q_2)=p^2,
$$
with distinct prime ideals $\p,\q_1,\q_2$. In particular, there is a prime ideal of $\O_K$ of norm $p$, as required.\\\\As all five rows now satisfy the hypotheses of Prop.\ \ref{prop:ext-progression}, the claim follows.
\end{proof}


\begin{thebibliography}{99}

\bibitem[A26]{Axiom26} Axiom Math Team, \emph{The Extended Riemann Hypothesis for Dedekind zeta functions of totally real fields}, Hexagon preprint 2611.01389v1 (29 November 2026) \url{https://hexagonmath.org/2611.01389v1}

\bibitem[Cla96]{Clark96} D. A. Clark, Non-Galois cubic fields which are Euclidean but not norm-Euclidean, {\it Math. Comp.} {\bf 65} (1996) 1675--1679

\bibitem[Gro26]{grobner} H. Grobner, {\it On divergence related to Riemann--von Mangoldt's explicit formula of the prime-counting function}, preprint, arXiv:2609.02713v2 (2026)

\bibitem[Har-Mur04]{HM}M. Harper, M. R. Murty, Euclidean rings of algebraic integers, {\it Canadian J. Math.} \textbf{56} (2004) 71--76

\bibitem[Has-She-Won22]{HSW} E. Hasanalizade, Q. Shen, P.-J. Wong, Counting zeros of Dedekind zeta functions, {\it Math. Computation} \textbf{91} (2022) 277--293

\bibitem[HKMS24]{HKMS} D. Hu, I. Kaneko, S. Martin, C. Schildkraut, On a Mertens-type conjecture for number fields, {\it Math. Proc. Cambridge Philos. Soc.} \textbf{177} (2024), 481--523

\bibitem[Iwa-Kow04]{IK04} H. Iwaniec, E. Kowalski, \emph{Analytic Number Theory}, American Mathematical Society Colloquium Publications {\bf 53} (AMS, Providence, RI, 2004)

\bibitem[Liu26a]{Liu26-ERH} B. Liu, \emph{The Riemann Hypothesis for Dedekind Zeta Functions of Totally Real and CM Fields}, Hexagon preprint 2611.01371v3 (1 December 2026) \url{https://hexagonmath.org/2611.01371v3}

\bibitem[Liu26b]{Liu26} B. Liu, {\it Zero-free half-planes for Hecke $L$-functions over number fields and Quasi-Riemann Hypothesis}, preprint, \href{https://arxiv.org/abs/2610.26835v2}{arXiv:2610.26835v2} (2026)


\bibitem[Mur-Pet13]{MP} M. R. Murty, K. L. Petersen, The Euclidean algorithm for number fields and primitive roots, {\it Proc. Amer. Math. Soc.} {\bf 141} (2013) 181--190

\bibitem[Nar07]{nark} W. Narkiewicz, Euclidean algorithm in small Abelian fields, {\it Funct. Approx. Comment. Math.} {\bf 37} (2007) 337--340

\bibitem[Neu99]{Neukirch} J. Neukirch, \emph{Algebraic Number Theory}, translated by N.~Schappacher, Grundlehren math. Wissenschaften {\bf 322} (Springer, Berlin, 1999).

\bibitem[OAI26]{OAI} OpenAI, \emph{Primitive roots for every admissible integer base}, preprint, Oct 4, (2026) \href{https://github.com/openai/math/blob/main/preprints/Primitive-roots-for-every-admissible-integer-base-October-4-2026/primitive-roots-all-integer-bases.pdf}{OpenAI mathematics repository.}

%\bibitem[OAI26b]{PD} OpenAI, \emph{The Poisson--Dirichlet law for prime predecessors}, preprint, September~24, (2026) \href{https://github.com/openai/math/blob/main/preprints/The-Poisson-Dirichlet-Law-for-Prime-Predecessors-September-24-2026/paper.pdf}{OpenAI mathematics repository.}

%\bibitem[OAI26c]{Even} OpenAI, \emph{Prime Predecessors with an Even Number of Prime Factors}, preprint, September~17 (2026) \href{https://github.com/openai/math/blob/main/preprints/Prime-Predecessors-with-an-Even-Number-of-Prime-Factors-September-17-2026/paper.pdf}{OpenAI mathematics repository.}

\bibitem[Sri-Sub-San23]{SSS23} K. Srinivas, M. Subramani, U. K. Sangale, Euclidean algorithm in Galois quartic fields, {\it Rend. Circ. Mat. Palermo (2)} {\bf 72} (2023) 1--7

\bibitem[Wei73]{wein} P. J. Weinberger, \emph{On Euclidean rings of algebraic integers}, in: \emph{Analytic Number Theory} Proc. Symp. Pure Math. {\bf 24} (AMS, Providence, RI, 1973) 321--332


\end{thebibliography}
\end{document}